\documentclass[12pt]{article}
\usepackage{amsmath}
\usepackage{amsfonts}
\usepackage{amssymb}
\usepackage{amsthm}

\usepackage{color}

\newtheorem{thm}{Theorem}[section]

\newcommand{\R}{{\rm I}\kern-0.18em{\rm R}}
\newcommand{\1}{{\rm 1}\kern-0.25em{\rm I}}
\newcommand{\E}{{\rm I}\kern-0.18em{\rm E}}
\newcommand{\p}{{\rm I}\kern-0.18em{\rm P}}
\makeatletter

\usepackage{epsfig}

\usepackage{fancyhdr}

\usepackage{graphics}

\usepackage{amssymb}
\usepackage{amsxtra}
\usepackage{amsfonts}

\usepackage{afterpage}
\usepackage{eucal}
\usepackage{eufrak}

\usepackage{enumerate}

\allowbreak

\author{Lev B Klebanov\footnote{Department of Probability and Statistics, MFF, Charles University, Prague-8, 18675, Czech Republic, e--mail: levbkl@gmail.com},  Irina Volchenkova\footnote{Department of Probability and Statistics, MFF, Charles University, Prague-8, 18675, Czech Republic, e--mail: i.v.volchenkova@gmail.com}}
\title{Heavy Tailed Distributions in Finance: Reality or Myth? Amateurs Viewpoint. }
\date{}

\begin{document}
\maketitle

\begin{abstract}

The purpose of this paper is to show that the use of heavy-tailed distributions in financial problems is theoretically baseless and can lead to significant misunderstandings. The reason for this the authors see in an incorrect interpretation of the concept of the distributional tail. In accordance with this, in applications it is necessary to use instead of the distributional tail the "smearing" of its central part.

\noindent
{\bf keywords}: heavy tailed distributions, exponential tails, pre-limit theorems, financial indexes
\end{abstract}

The authors are not yet experts in the field of finance. However, we are interested in the use of probabilistic methods in various areas, including the financial Economics. It seems to us that the use of some concepts of probability theory in financial mathematics is baseless or, at least, insufficiently substantiated. In this article we are talking about the use of heavy-tailed distributions. Let us describe our doubts in more details.

Let us start with commonly offered arguments in favour of using heavy-tailed distribution in the description of the stochastic noise, which arises in the study of changes in financial indices and/or prices. Straight away we will say, from our point of view these arguments do not confirm the presence of heavy tails. But let us explain step by step.

\section{First argument}\label{sA1}
\setcounter{equation}{0}

The first argument usually arises when considering some of the time series, such as the Dow Jones Industrial Average index (say, for the interesting Period from the 3 January 2000 to 31 December 2009), daily ISE-100 Index (November 2, 1987 - June 8, 2001) and many others (see, e.g., \cite{EK}, \cite{BMW}). The observed fact is that quite a lot of data not only fall outside the 99\% confidence interval on the mean, but also outside the range of $\pm 5\; \sigma$ from the average, or even $\pm10\;\sigma$. On the assumption of this circumstance we can make two conclusions.

First (and absolutely correct) conclusion consists in the fact, that the observations under assumption of their independence and identical distribution are in contradiction with their normality.

The second conclusion is that the distribution of these random variables is heavy-tailed. This decision is not based on any mathematical justification. Indeed, the first thing that comes to mind in this situation is to try to apply the Chebyshev`s inequality to the random variables with non-normal distributions (such way of proof that random variables don`t follow normal distribution is given in \cite{BMW}). An exact inequality can be found in the book  Karlin and Studden (\cite{KS}). There is also shown distribution, when inequality becomes equality. To quote the corresponding result: 
{\it Gauss Inequality. In this example we desire to determine 
the maximum value of 
\[ \p\{X \in (-\infty, \mu - d] \cup [\mu + d, \infty) \}, \;  d > 0 \]
over the class of unimodal distribution functions with mode and mean 
located at $\mu$ and variance equal to $\sigma^2$. 
The solution for the case $d^2 \geq 4 \sigma^2 /3$ is  given by 
$4 \sigma^2 /9 d^2$, with rectangular part of the distribution on interval $[\mu-3 d/2, \mu + 3 d/2]$ plus 
mass at $\mu$. By a rectangular distribution on $[a, b]$ we mean a  distribution $F$ whose density is 
$1/(b-a)$ for $ x\in (a,b)$ and $0$ otherwise. }

Let us apply this result to the studied case supposing that the distribution is unimodal with finite variance (e.g., $\sigma =1$), i.e., not a heavy-tailed distribution. We choose $\mu=0$, $\sigma =1$ and $d=10 \sigma =10$. From the earlier mentioned follows
\[  \p \{ |X| >10 \} \leq \frac{1}{225},  \]
and the equality is reached for the above-written distribution.
It is clear that the probability of $1/225$ is not so small. Exactly in the case of a sample  size  of 50000 average number of exceeds of level $10 \sigma$ is more than 222 times. Notice that  samples of 50000 are not uncommon in financial problems. Moreover, in sample such as this with $d=40\sigma =40$, an average number of exceeds of level $40\sigma$ will be somewhat more 13.8. It is clear, that in this case we are not talking about the heavy tail ($\sigma =1$ !). Thus the first argument of the appearance of heavy-tailed distributions is rejected.

Of course, used extremal distribution does not seem to be natural for describing the fluctuations of financial indices. Especially strange looks existence of the mass at $\mu$. Therefore, we will give an example  of continuous distribution, which seems to be more natural in this case. In addition, this distribution will give a counterexample to the other arguments about the presence of heavy tails in financial problems.

Let $Y$  be a random variable with a gamma distribution with the shape parameter $1/m$ and the scale parameter $m$, where $m$ is the whole positive number. Assume $Y_1, Y_2$ are two independent variables and follow the \\ distribution of $Y$. Define $X$ as $Y_1-Y_2$, this is the random variable we need. It follows symmetrized gamma distribution with the characteristic function $1/(1+mt^2)^(1/m)$
and the probability density function
\[  p_m(x) =\frac{2^{1/2-1/m} |x|^{1/m-1/2}}{\sqrt{\pi} \Gamma(1/m)m^{(2+m/(4m))}}K_{1/m-1/2}(|x|/\sqrt{m}), \]
where $\Gamma(z)$ is the gamma function, and $K_{\nu}(z)$ is the modified Bessel function of the second kind. 

It is easy to see, that this distribution has an exponential tail and hence there exist finite moments of all orders. The variance of this distribution is equal to  two (i.e., $\sigma = \sqrt{2}$) for all  $m>0$. Notice that this distribution has an interesting property of $\nu$-normality, i.e., it plays the role of a normal distribution for the summation of a random number of random variables (see \cite{KKRT}). This will be discussed in more details later. 

Table \ref{tab1} gives the probability of deviations of the random variable $X$ larger than $10 \sigma$ from the average. We can see this probability incleases with $m$.

\begin{table}
\caption{ Probabilities of deviations larger than $10 \sigma$}\label{tab1}
\begin{center}
\begin{tabular}{|l|l|l|l|}
\hline
m & p & m & p\\
\hline
10 & 0.0000589843& 60 & 0.000757375\\
\hline
20 &  0.000230141& 70 & 0.000833146\\
\hline
30 & 0.000401799 & 80 &   0.000894442\\ 
\hline
40 &  0.000546297 & 90 & 0.000944249\\
\hline
 50 &  0.000663305 & 100 & 0.000984872\\
\hline
\end{tabular}
\end{center}
\end{table}
  
\section{Second argument}\label{sA2}
\setcounter{equation}{0}

The second argument is that the observed kurtosis exceeds 3 and sometimes quite significantly (see. \cite{BT, DuPe}). From our point of view this argument sounds strange. Indeed, it is known (see e.g., \cite{Rao}), that the asymptotic dispersion of empirical moments depends on the higher-order moments, which typically have high values. Thus for better approximation of the theoretical (general) moments, it is necessary to have a large number of the observations.  Therefore, the high value of empirical kurtosis does not essentially mean that the theoretical kurtosis is anomally high or infinite.

However, it seems more important  for us that for infinitely divisible distributions the kurtosis either does not exist, or is at least three, and equality holds only in case of normal distribution. This fact was probably known by the experts, but we could not have found where and who it had been published by, therefore give its exact formulation  and proof in the appendix. Since financial indices are related to the sums of a large number of random variables, the assumption of infinite divisibility of the corresponding distribution sounds quite natural. And then an observed fact is simply in agreement with this assumption.  And it does not carry any further information on the existence of heavy tails.

Let us see how kurtosis depends on the parameter $m$ of the symmetrized gamma distribution. It is easy to calculate that the kurtosis is expressed by the formula $\kappa =3(1+m)$.
We see that with an increase of $m$, the kurtosis increases. However, as it has been already noted, the tail of the distribution is not heavy. Thus, the second argument also does not prove the presence of heavy tails of the corresponding distribution.

\section{Third argument}\label{sA3}
\setcounter{equation}{0}

The third argument \cite{Vid2010} is that for values of $x$ from some broad interval a ratio $\p\{X>x\}/\p\{X>1.5 x\}$  for the empirical distribution of a financial index is close to a constant. This argument, of course, would speak about the tails weight, if it was for all sufficiently large $x$ and the theoretical distribution of the random variable $X$. However, the tails of the empirical distribution vanish from some point, so it is not clear what values should be considered "sufficiently large". If just consider some interval, the above property may be valid for distributions without heavy tails. For example, if the variable $X$ is uniformly distributed on the interval $(-A,A)$ with a large $A>0$, then the specified  ratio is constant at $x,\; 1.5x \in (-A,A)$ for the theoretical distribution.

On the picture \ref{fig1} there is given a plot of the ratio  $\p\{X>x\}/\p\{X>1.5 x\}$ for symmetrized gamma distribution, $m=50$, over interval $(1, 50)$.
From the plot it is clear, that this ratio is almost a constant on the wide interval. Let us remind, that the variance $\sigma^2$ of the explored distribution is equal to 2, so that the length of the given interval  is about $35\sigma$. Here again we speak about the theoretical distribution. For empirical distribution it is not clear, what length of the interval is  sufficient, and thus, what sample size is required.

As two previous arguments, the third one is not an argument  in favour of the heavy tails.

\begin{figure}[h]
\centering
\hfil
\includegraphics[scale=0.5]{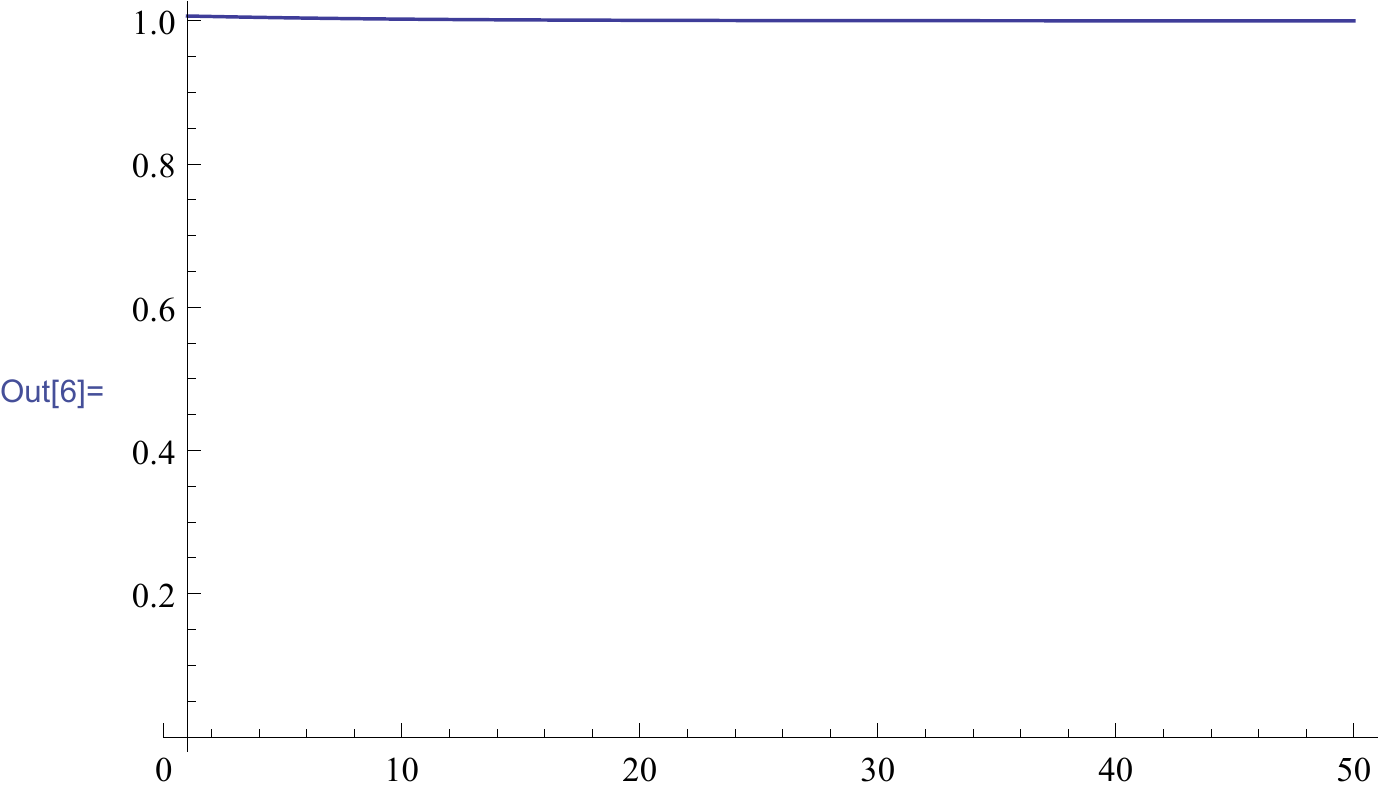}
\caption{Plot of the ratio  $\p\{X>x\}/\p\{X>1.5 x\}$ for symmetrized gamma distribution, $m=50$, over interval $(1, 50)$ }\label{fig1}
\end{figure}

\section{Fourth argument}\label{sA4}
\setcounter{equation}{0}
The fourth argument is that the evaluation of the index of the tail, for example,  Hill  estimator and/or its modifications, give a value between 1 and 5 \cite{ChDLNB}, i.e. power law tail. However, this argument does not take into account several factors. First of all, any statistical evaluation of the tails index requires a large and unknown number of the observations. Although it is a consisten estimator, but for a given accuracy the number of observations can not be uniformly estimated in principle on a large enough class of distributions. 

Indeed, we can not know anything about the behavior of the tail after receiving the maximal value of the observation, i.e. this behavior can be anything. The hope is that we still have a sufficient number of observations simply fading away for the case where the value of estimate is less than 2. 

In fact, in this case the variance of the theoretical distribution is infinite, which means that it must be large for the empirical distribution. However, in examples of the financial indices an empirical variance is mostly around 1, i.e,, not large. This fact, of course, does not talk about the existence of the finite variance of the theoretical distribution, but says that the sample size is insufficient for the construction of a "good" variance estimation. So why the sample of this size can be sufficient for receiving a good estimate of the index of the tail? We have never seen the arguments on this subject.

We checked out our arguments by simulating samples from symmetrized gamma distribution with parameter $m=10$. The sample size was $n=10000$. The parameter $k(n)$ of the Hill  estimator was selected in three ways:
$k(n)=\sqrt{n}$, $k(n) = n^{2/3}$ and $k(n)=n^{4/5}$. Average value (for hundred simulations) of the Hill estimator were, respectively, $0.37$, $0.65$ and $1.39$. All these values, which should indicate the presence of very heavy tail, of course, do not reflect its exponential character for symmetrized gamma distribution.
Thus, we reject fourth argument in favour of heavy-tailed distributions.

It is necessary to keep in mind, that the results of asymptotic behavior of Hill's estimator (and in general of all the estimators of indices of tail distribution) assume that random sample follows the distribution with power tail. What result can the one expect, if this assumption is not satisfied? Apparently, anything, because there is no results regarding this case and can not be, since the tail can change its behavior from the very far point. The values of Hill`s estimation can be similar to the expected values of heavy-tailed distribution, but it does not prove anything and the tail can be not heavy. It should be noted that even newer and more accurate estimates of the index (see e.g., \cite{NgSam}) does not prove the presence of a heavy tail, as the behavior of these estimates in the "alternative" distribution is unknown.

Thus, we reject fourth argument in favour of heavy-tailed distributions.

\section{Fifth argument}\label{sA5}
\setcounter{equation}{0}

The fifth argument is that the type of distribution of financial indices should be associated with sums of a large number of random variables, and thus should be close to the type of one of the limiting distributions. According to the central limit theorem, under specific and well-known conditions, the limiting distribution may be either normal or stable with the stability parameter $\alpha \in (0, 2)$. The normal distribution has to be rejected in accordance with the first argument (and we agree with this), and then there is left just some of the stable distributions. All stable non-normal distributions have heavy tails.

Sometimes the fifth argument is a little bit modified: instead of referring to the central limit theorem, the property of stability itself is used, i.e. the fact of the similarity of the distributions of big and smaller sums of random variables.

However, again we find this argument insufficiently and not-well sounded. In fact, when using  the central limit theorem in its simplest form, it is necessary to have the sum of independent identically distributed random variables. This sum must be properly centered and standardized. Of course, there exist modifications for cases of differently distributed and (weakly) dependent components. But it is hard to check out these "generalized" conditions for empirical data.  It is also not clear why the components, that make up the index values, are identically distributed and independent. It is not clear, what role plays standardization. From the practical point of view, it is not necessary to answer these questions, but consider the limit theorem only as a hint that facilitates the choice of the distribution, and then check, how it corresponds with the reality. However, this point of view leads us rather to the normal distribution than to stable. Why is that? The explanation is simple. Any financial values are bounded in view of the fact that in the world there is a large, but finite amount of money. And it is not clear, why the heavy-tailed distribution should be a good approximation.  When applying the limit stable convergence theorem it is necessary the components were heavy-tailed distributed themselves. Why can they arise in problems with bounded random variables? Examples of such distributions have already arisen in physics, but not in finances.

We believe that the references to the limit theorems on confirmation of the empirical data should be treated carefully. Here is an example of the completely wrong conclusions as a result of the inattentive analysis. We were simulating 1000 samples of  size $n=10000$. Each sample followed symmetrized gamma distribution with parameter $m=100$. For each sample was calculated the sum of the random variables and then each sum was standardized with respect to  $1/n^{1/1.83}$,  i.e. as in the case of the stable distribution with parameter $\alpha = 1.83$. It is obvious, that the limit theorem is not applicable in this situation, because the observations do not belong to the domain of attraction of an $\alpha$-stable distribution. However, the empirical distribution function of the standardized sums is very similar to a symmetric stable distribution with parameters  $\alpha = 1.8$ and $\sigma=6$. The one can verify it by looking at the chart \ref{fig2}.

\begin{figure}[h]
\centering
\hfil
\includegraphics[scale=0.5]{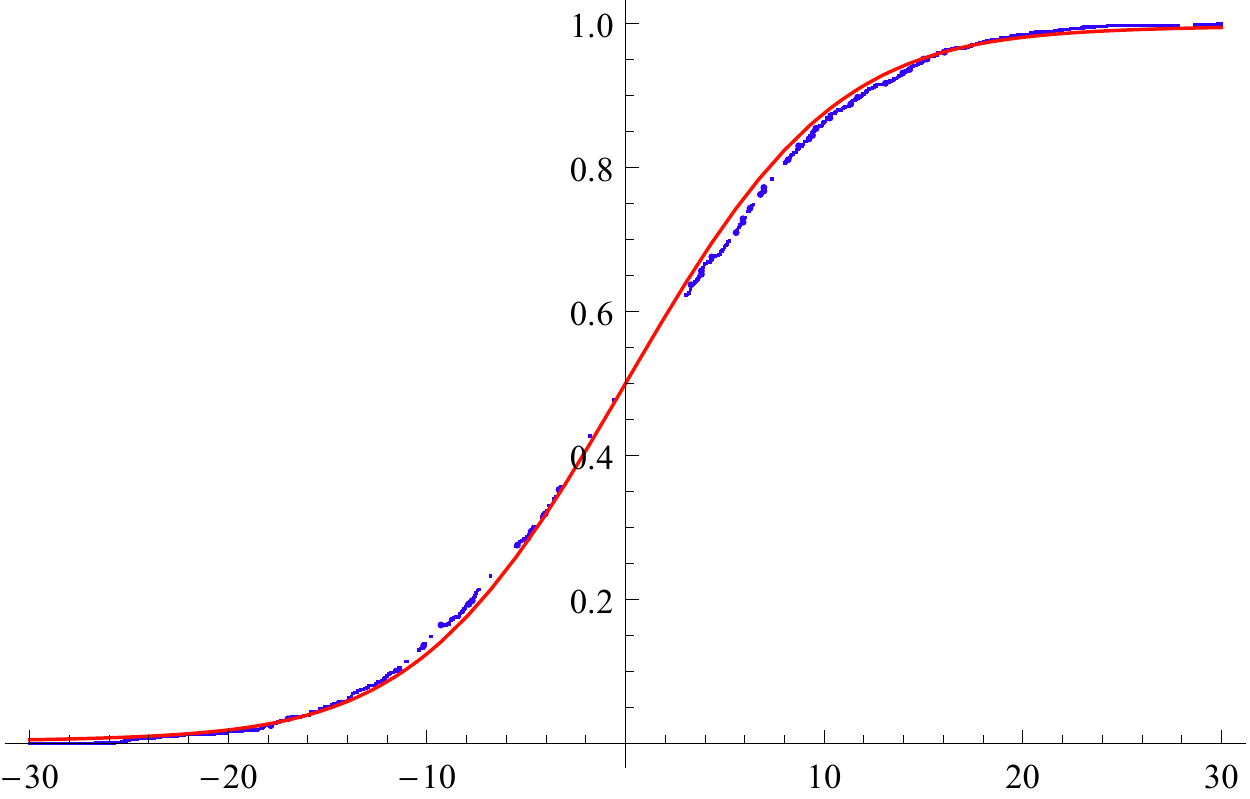}
\caption{Plot of empirical CDF of normalized sum (blue) versus CDF of symmetric stable distributions, $\alpha = 1.8$ and $\sigma=6$ (red)}
\label{fig2}
\end{figure}
The reason for this similarity consists in the fact, that we took into account only distributions behavior around the center area, where they are very similar, but did not take into consideration the "far" tails behavior (where distribution functions vary greatly). In this case the corresponding sums behavior can be described by the pre-limit theorem, not by limit theorem \cite{KRSz}.

In our opinion, it is reasonable to refer to the limit results of random sums of random variables, but not to the central limit theorem. In fact, the number of transactions per unit of time on the market is a random variable. The number of stocks in a portfolio is also a random variable.  Notice that the analogues of normal and stable random variables were implemented and studied in \cite{KMM84}, \cite{KMM87},\cite{KR96}, \cite{KKRT}. These distributions differ from the classic stable distributions, but have the property of stability (in which instead of the sum of random variables is used the sum of a random number of components). There also exist analogues of the corresponding  limit theorems. In particular, as it has already been noted, symmetrized gamma distribution is an analogue of a normal distribution. Let us state this fact accurately (see \cite{KKRT}, \cite{KKozR}).

Let $X_1, X_2 \ldots $ be a sequence of independent identically distributed random variables. Assume that $\{ \nu_p, \; p\in (0,1)\}$ is a family of random variables, independent of the $X_1, X_2 \ldots $  with the probability generating function 
\begin{equation}\label{eqGF}
{\mathcal P}_p(z,m) =\frac{p^{1/m}z}{(1-(1-p)z^m)^{1/m}},
\end{equation}
where $p\in (0,1)$, and $m$ is positive whole-number parameter. We have $p = 1/\E\nu_p$. Thus, if the sum $S_{p}= \sum_{j=1}^{\nu_p}X_j$ includes in average many components, then  $p$ is small. We say, that the random variable $X_1$ strictly follows $\nu$-normal distribution, $\E X_1 =0$ and
\[ X_1 \stackrel{d}{=}p^{1/2}\sum_{j=1}^{\nu_p}X_j \]
for all $p \in (0,1)$.
The symbol $ \stackrel{d}{=}$ means equality of distributions. Of course, symmetrized gamma distribution is the strictly  $\nu$-normal distribution and it holds the next limit theorem (see \cite{KKRT}).

\begin{thm}\label{thLim}
Let $Y_1, Y_2, \ldots$ be a  sequence of independent identically \\distributed random variables. Assume $\E Y_1=0$, $\E Y_1^2< \infty$, and defined above family $\{ \nu_p, \; p\in (0,1)\}$ does not depend on $Y_1, Y_2, \ldots$. Let $S_{p}= \sum_{j=1}^{\nu_p}Y_j$. Then
\[ \lim_{p \to 0}\p\{S_p <x\} = \p\{X_1<x\}, \]
where the random variable $X_1$ has a strictly $\nu$-normal distribution.
\end{thm}
This theorem plays the role of the central limit theorem for random sums of random variables, when the number of components has a distribution with probability generating function (\ref{eqGF}). We have already seen that the symmetrized gamma distribution quite corresponds to all the previous arguments, though it does not have heavy tail. It also corresponds to the fifth argument, if use the central limit theorem in expanded form. The advantage also consists in the fact, that there is no need components to have heavy tails.

\section{The last sixth argument}\label{sA6}
\setcounter{equation}{0}

The argument is that the use of stable distributions, or hyperbolic distributions, or other distributions with heavy tails fits well the real data (see articles in \cite{Handbook} and also \cite{GU}, where the reasons to use hyperbolic, but not stable distributions, are given). Partially, it is true. But it is well known, that a good fit does not guarantee the accuracy of the model used. Historical examples of this type are known. Really, heliocentric model of the solar system, at the very beginning of its development, or the Ptolemaic model. At that time, the second one described astronomical events more accurately comparing to the first one. However, this did not prevent the heliocentric model prevail. In addition, we can see, that there arise modifications of stable distributions, for instance, tempered stable distributions \cite{Ros2007, RosSin,BMW}. These distributions have exponential tails and give a better fit than stable. However in this case, the better fit can easily be explained by a large number of parameters. In addition, it makes sense to modify the model, which explains some  basic reality facts, but for some other needs to be clarified. 

From all the arguments above, it follows that the model with heavy tails does not explain the fundamental phenomena and facts, so it is not clear, if such a model should be modified, or it`s more correct to propose new ones. In fact, it is not even clear, what heavy-tailed models should be chosen for modification. After all there are used  models of stable distributions, hyperbolic distributions, distributions of extreme values, and some others. The goodness of fit of the proposed symmetrized gamma distribution model is not checked yet on the real data, but we think it deserves such a verification.

Note, that a very critical point of view of the arguments about an existence of agreement with already observed data without adequate abilities to predict, has already been expressed with regard to financial problems in \cite{BBPZ}.

\section{Our offer and final remarks}\label{sC}
\setcounter{equation}{0}
At all the written above is possible (and expected) simple reaction: "All these arguments are only theoretical and have a critical nature. But we must somehow process the financial data in the absence of normality. There are offered many models with heavy tails, and they often work. And you do not offer anything." This reaction is partially correct, because we have not yet conducted an analysis on the compliance between the proposed model and the real data. However, an alternative model itself is offered. It is symmetrized gamma distribution. Below we present some of its properties and explicitly give some more theoretical arguments in its defense.

We will only consider the symmetric case. This does not mean that the model can not be simply replaced by a more general. We will present one of the possible not symmetrized versions lower, but will not study it.

So in the symmetric case we offer the family of distributions with the characteristic functions 
\begin{equation}\label{eqG}
f(t,m)=\frac{1}{(1+m t^2)^{1/m}},\; \;t \in \R^1,
\end{equation}
where $m>0$ is a parameter. This is the characteristic function of symmetrized gamma distribution in standardized form. 
Exactly its expected value and dispersion are equal to zero and resp. to two. Of course, it is possible to implement a location parameter and a scale parameter, but we will study the standard case.

As it has already been mentioned in the section \ref{sA5}, the distribution with characteristic function (\ref{eqG}) in the case of whole positive $m$ is an analogue of the normal distribution, when summarizing the random number of independent identically distributed random values.\footnote{The similarity with the property of stability}. Of course, this characteristic function for any $m>0$ is analytic in the strip $|Im t|< 1/\sqrt{m}$, and thus has exponentially decaying tails, and, therefore, the finite moments of all orders \footnote{difference from the distributions with heavy tails}. 
In an obvious way, the corresponding distribution is infinitely divisible. \footnote{i.e. it can be considered as distribution of the sum of arbitrarily large number of independent variables}. It also clear, that its kurtosis $ \kappa> 3 $ for all $ m> 0 $. 

It is easy, however, to write an exact expression for $ \kappa $. That is, $\kappa =3(1+m)$. Is it arbitrarily large with appropriate $m$. \footnote{Again the similarity with what  was expected from the heavy-tailed distributions} Some properties have already been noted in the previous section. We note some new facts that explain similarities with heavy-tailed distributions. 

First of all, let us notice one very usable property of the distribution with the characteristic function (\ref{eqG}). If $X_1, X_2, \ldots , X_n$ are independent identically distributed random variables with the characteristic function (\ref{eqG} and
\begin{equation}\label{eqS}
S_n=\frac{1}{\sqrt{n}}\sum_{j=1}^{n}X_j, 
\end{equation}
then the characteristic function of $S_n$ has form
\begin{equation}\label{eqCFS}
 f^{n}(t/\sqrt{n})= f(t, m/n). 
 \end{equation}

It shows, that the summation of independent identically distributed random variables with the distribution (\ref{eqG}) leads to the change of the parameter $m$ of this distribution.
The specified property also somewhat similar to the property of the stability, although it is not stability property itself. We also note that for all $m>0$ and $t\in \R^1$ is satisfied an inequality
\begin{equation}\label{eqPar}
f(t,m) \geq \exp\{-t^2\} 
\end{equation}
with equality only for $t=0$. However it is clear, that $f(t,m) \to \exp\{-t^2\}$ with $m \to \inf$ by the central limit theorem. On the other hand, the ratio (\ref{eqCFS}) shows, that similarity between the distribution of sum (\ref{eqS}) and normal distribution depends on the quotient $n/m$.

In the paper \cite{KRSz} was shown, that the behavior of the characteristic function on the interval $(\delta, \Delta)$, where $\delta >0$ is small, defines the behavior of distribution of the sums of independent random values with the large number of summands $n$ depended on the $\delta$ and $\Delta$. Moreover, if the characteristic function is close to the stable on this interval, then the distribution of the sum is also close to the stable for slightly high $n$ (the upper bound of the interval depends on $\delta$ and $\Delta$). Let us try to find the characteristic function which has the following form $\exp\{-\lambda |t|^{\alpha}\}$ (i.e. stable) and for $t \in (\delta, \Delta)$ satisfies
\begin{equation}\label{eqIn}
\frac{1}{(1+mt2)^{1/m}} > \exp\{-\lambda t^{\alpha}\}> \exp\{-t^2\}. 
\end{equation}
In this case the stable distribution fits the "middle" tail of the symmetrized gamma distribution better than normal. Now show that this choice of the parameters $\lambda >0$ and $\alpha \in (0,2)$ is always possible. Really, the inequality (\ref{eqIn}) is equivalent to the 
\begin{equation}\label{eqIn2}
\frac{\log(1+mt^2)}{m t^2} < \lambda t^{\alpha-2} <1,
\end{equation}  
with $t \in (\delta, \Delta)$. It is easy to see, that the function $\log(1+u)/u$ decreases for $u>0$, and more, it converges to the 1 with $u \to 0$.Therefore, for $t=\delta$ it is smaller than 1, and we can satisfy the inequality (\ref{eqIn2}) choosing $\lambda>0$ small enough at the point $t=\delta$. Obviously, it will be satisfied for $t>\delta$. 

We calculated estimations of the parameters $\lambda$ and $\alpha$ of the function (\ref{eqCFS}) with $m=20$, values $n$, from 1 to 100 with step size of 10 and $\delta=0.005$, $\Delta=0.5$. They results are given in the table \ref{tab3}.

\begin{table}
\caption{Values of stable distribution parameters $\alpha$ and $\lambda$}\label{tab3}
\begin{center}
\begin{tabular}{|l|l|l|}
\hline
n & $\alpha$ & $\lambda$ \\
\hline
1 &1.26906 & 0.226565\\
\hline
10& 1.80697 & 0.720949\\
\hline
20 & 1.89192& 0.835951\\
\hline
30 & 1.92487 & 0.883786\\
\hline
40 & 1.94241 & 0.910016\\
\hline
50 & 1.9533 & 0.926583\\
\hline
60 & 1.96073 & 0.937998\\
\hline
70 & 1.96612 & 0.946341\\
\hline
80 & 1.97021 & 0.952704\\
\hline
90 & 1.97341 & 0.957718\\
\hline
100 & 1.976 & 0.961771\\
\hline
\end{tabular}
\end{center}
\end{table}

These values can be interpreted as follows. For a sample of size $n$ from the distribution (\ref{eqG}) the characteristic function (and, hence, the distribution of the standardized sum (\ref{eqS}) is fitted by the stable distribution with parameters specified in the table better than with the normal distribution. This approximation was close to the optimal. This implies a very interesting conclusion. It consists in the fact that the "extension" of the data, i.e. the replacement with sums of some numbers of observations, the parameter $\alpha$ (parameter of the distribution used for fit) grows (which would not occur in the case of stable distribution). With regard to the financial data, it would mean that the estimates of the tail index from data obtained at a short and some longer time intervals  must be different. This difference actually exists. Attempt to explain it on the basis of pre-limit theorem has been made in the work \cite{GrSam}.

As it has been mentioned, one can implement asymmetrical versions of the above-studied distribution (\ref{eqG}). These versions differ in the way of centering the sum of the random number of random values (\ref{eqS}). The first one consists in centering each summand in the sum separately and in the second way the whole sum is centered. Two methods can be applied simultaneously. We will not discuss the details of the corresponding methods, in our opinion, it should be determined separately in each specific task.

Note also, that the fact of $\nu$-normality of the considered distribution means that it can be viewed as dispersion-based mixture of the normal distributions. Mixture distribution is the gamma distribution. In other words, a random variable following gamma symmetrized distribution can be considered as a random choice from the normal sample with random variance.

Unfortunately we can not yet give final answer to the question, posed in the title. However, we believe, that the presence of the heavy tails in finances is myth. The concept of distribution tail itself should be modified in applied researches. It seems, that it should be replaced with the behavior of characteristic function on the interval $(\delta, \Delta)$ or with the properties of distant part of the distribution center.

\section{Appendix}
\setcounter{equation}{0}

\begin{thm}\label{thA}
Let assume that $X$ is infinitely divisible random variable with a finite moment of the fourth order $\mu_4$. Let $\kappa = \mu_4/\mu_2^2$ be kurtosis of the random variable $X$. Then
\[ \kappa \geq 3 \] and equality occurs if and only if $X$ has a normal distribution.
\end{thm}

\begin{proof} We consider only the case of a symmetric distribution. The general case is similar, but requires more difficult calculations. Let $X$ be a random variable with infinitely divisible distributions. We denote $f(t)$ its characteristic function and let $\varphi (t) = \log f(t)$. Since $X$ has a finite kurtosis, then there is a finite fourth moment of the distribution of  $X$, i.e. there exist derivatives of functions $f(t)$ and $\varphi(t)$ up to fourth order. Kolmogorov canonical representations shows that
\[  \varphi (t) = ict + \int_{-\infty}^{\infty} (e^{itx}-1-itx)\frac{d K(x)}{x^2}. \]
And it follows, that 
\[ \varphi^{\prime\prime}(t) =-\int_{-\infty}^{\infty} e^{itx} dK(x), \]
\[ \varphi^{(IV)}(0) =\int_{-\infty}^{\infty} x^2 dK(x) \geq 0. \]
And the equality in the last inequality is reached if and only if the distribution  $K$ is singular at point $x=0$. It is clear, that in this (and only this) case $f(t)$ is a characteristic function of normal distribution $\varphi^{(IV)}(0) = f^{(IV)}(0) -3( f^{\prime \prime})^2(0)$. Further, it is easy to verify that for symmetric distribution $\varphi^{(IV)}(0) = f^{(IV)}(0) -3( f^{\prime \prime})^2(0)$.
Thus, from the inequality $ \varphi^{(IV)}(0) > 0$, it follows 
\[   f^{(IV)}(0) -3 (f^{\prime \prime})^{2} (0) \>0.\]

If we denote $\mu_j$ the  $j$-th moment of a random variable $X$, then the preceding inequality can be written as \[ \mu_4 - 3\mu_2^2 >0. \]
\end{proof}

\end{document}